\documentclass[11pt,a4paper]{amsart}[2000]
\usepackage{geometry}
\usepackage{amsmath}
\usepackage{amssymb}
\usepackage{amsthm}
\usepackage[hang]{footmisc}
\usepackage[all]{xy}
\usepackage{hyperref}
\usepackage{color}
\usepackage{mathtools}
\usepackage{enumitem}
\usepackage{stackengine}

\title[On Amenability and LEF-ness for topological full groups]{On Amenability and LEF-ness for topological full groups}

\thanks{}

\allowdisplaybreaks

\theoremstyle{plain}

\newtheorem{Thm}{Theorem}[section]

\theoremstyle{definition}

\theoremstyle{plain}
\newtheorem{thm}[Thm]{Theorem}
\newtheorem{lem}[Thm]{Lemma}
\newtheorem{cor}[Thm]{Corollary}
\newtheorem{prop}[Thm]{Proposition}

\theoremstyle{definition}
\newtheorem{defn}[Thm]{Definition}
\newtheorem{ques}[Thm]{Question}

\newtheorem{rmk}[Thm]{Remark}

\newtheorem{innercustomthm}{Theorem}
\newenvironment{customthm}[1]
{\renewcommand\theinnercustomthm{#1}\innercustomthm}
{\endinnercustomthm}

\newcommand{\A}[0]{\mathbb{A}}
\newcommand{\B}{B}
\newcommand{\J}{J}
\newcommand{\K}{\mathcal{K}}
\newcommand{\D}{D}
\newcommand{\Ch}{D}
\newcommand{\Zh}{\mathcal{Z}}
\newcommand{\E}{E}
\newcommand{\Oh}{\mathcal{O}}

\newcommand{\T}{{\mathbb T}}
\newcommand{\R}{{\mathbb R}}
\newcommand{\N}{{\mathbb N}}
\newcommand{\Z}{{\mathbb Z}}
\newcommand{\C}{{\mathbb C}}
\newcommand{\Q}{{\mathbb Q}}
\newcommand{\F}{{\mathbb F}}

\newcommand{\aut}{\mathrm{Aut}}
\newcommand{\supp}{\mathrm{supp}}

\newcommand{\eps}{\varepsilon}
\numberwithin{equation}{section}
\newcommand{\map}[0]{\operatorname{Map}}

\newcommand{\id}{\mathrm{id}}

\newcommand{\halpha}{\widehat{\alpha}}
\newcommand{\calpha}{\widehat{\alpha}}

\newcommand{\tih}{\widetilde {h}}

\newcommand\set[1]{\left\{#1\right\}}  
\newcommand\mset[1]{\left\{\!\!\left\{#1\right\}\!\!\right\}}

\newcommand{\CA}[0]{\mathcal{A}} \newcommand{\CB}[0]{\mathcal{B}}
\newcommand{\CC}[0]{\mathcal{C}} \newcommand{\CD}[0]{\mathcal{D}}
\newcommand{\CE}[0]{\mathcal{E}} 
\newcommand{\CG}[0]{\mathcal{G}} \newcommand{\CH}[0]{\mathcal{H}}

\newcommand{\CO}[0]{\mathcal{O}} \newcommand{\CP}[0]{\mathcal{P}}
\newcommand{\CQ}[0]{\mathcal{Q}} 
 \newcommand{\CT}[0]{\mathcal{T}}
\newcommand{\CU}[0]{\mathcal{U}}

\newcommand{\Ra}[0]{\Rightarrow}
\newcommand{\La}[0]{\Leftarrow}
\newcommand{\LRa}[0]{\Leftrightarrow}

\newcommand{\quer}[0]{\overline}
\newcommand{\eins}[0]{\mathbf{1}}			
\newcommand{\diag}[0]{\operatorname{diag}}
\newcommand{\ad}[0]{\operatorname{Ad}}
\newcommand{\ev}[0]{\operatorname{ev}}
\newcommand{\fin}[0]{{\subset\!\!\!\subset}}
\newcommand{\diam}[0]{\operatorname{diam}}
\newcommand{\Hom}[0]{\operatorname{Hom}}
\newcommand{\dst}[0]{\displaystyle}
\newcommand{\spp}[0]{\operatorname{supp}}
\newcommand{\lsc}[0]{\operatorname{Lsc}}
\newcommand{\del}[0]{\partial}
\newcommand{\GU}[0]{\CG^{(0)}}

\newtheorem{lemma}[Thm]{Lemma}

\theoremstyle{definition}

\numberwithin{equation}{Thm}

\begin{document}
	\global\long\def\floorstar#1{\lfloor#1\rfloor}
	\global\long\def\ceilstar#1{\lceil#1\rceil}	
	
	\global\long\def\B{B}
	\global\long\def\A{A}
	\global\long\def\J{J}
	\global\long\def\K{\mathcal{K}}
	\global\long\def\D{D}
	\global\long\def\Ch{D}
	\global\long\def\Zh{\mathcal{Z}}
	\global\long\def\E{E}
	\global\long\def\Oh{\mathcal{O}}

	\global\long\def\T{{\mathbb{T}}}
	\global\long\def\BR{{\mathbb{R}}}
	\global\long\def\N{{\mathbb{N}}}
	\global\long\def\Z{{\mathbb{Z}}}
	\global\long\def\C{{\mathbb{C}}}
	\global\long\def\Q{{\mathbb{Q}}}

	\global\long\def\aut{\mathrm{Aut}}
	\global\long\def\supp{\mathrm{supp}}

	\global\long\def\eps{\varepsilon}

	\global\long\def\id{\mathrm{id}}

	\global\long\def\halpha{\widehat{\alpha}}
	\global\long\def\calpha{\widehat{\alpha}}

	\global\long\def\tih{\widetilde{h}}

	\global\long\def\opFol{\operatorname{F{\o}l}}

	\global\long\def\opRange{\operatorname{Range}}

	\global\long\def\opIso{\operatorname{Iso}}

	\global\long\def\dimnuc{\dim_{\operatorname{nuc}}}

	\global\long\def\set#1{\left\{  #1\right\}  }

	
	\global\long\def\mset#1{\left\{  \!\!\left\{  #1\right\}  \!\!\right\}  }

	\global\long\def\Ra{\Rightarrow}
	\global\long\def\La{\Leftarrow}
	\global\long\def\LRa{\Leftrightarrow}

	\global\long\def\quer{\overline{}}
	\global\long\def\eins{\mathbf{1}}
	\global\long\def\diag{\operatorname{diag}}
	\global\long\def\ad{\operatorname{Ad}}
	\global\long\def\ev{\operatorname{ev}}
	\global\long\def\fin{{\subset\!\!\!\subset}}
	\global\long\def\diam{\operatorname{diam}}
	\global\long\def\Hom{\operatorname{Hom}}
	\global\long\def\dst{{\displaystyle }}
	\global\long\def\spp{\operatorname{supp}}
	\global\long\def\spo{\operatorname{supp}_{o}}
	\global\long\def\del{\partial}
	\global\long\def\lsc{\operatorname{Lsc}}
	\global\long\def\GU{\CG^{(0)}}
	\global\long\def\HU{\CH^{(0)}}
	\global\long\def\AU{\CA^{(0)}}
	\global\long\def\BU{\CB^{(0)}}
	\global\long\def\CUU{\CC^{(0)}}
	\global\long\def\DU{\CD^{(0)}}
	\global\long\def\QU{\CQ^{(0)}}
	\global\long\def\TU{\CT^{(0)}}
	\global\long\def\CUUU{\CC'{}^{(0)}}
	
	\global\long\def\bA{\mathbb{A}}
	\global\long\def\AUl{(\CA^{l})^{(0)}}
	\global\long\def\BUl{(B^{l})^{(0)}}
	\global\long\def\HUp{(\CH^{p})^{(0)}}
	\global\long\def\sym{\operatorname{Sym}}
	
	\global\long\def\properlength{proper}

	\global\long\def\interior#1{#1^{\operatorname{o}}}
	
	\author{Xin Ma}
	\email{xma17@hit.edu.cn}
	\address{Institute for Advanced Study in Mathematics, Harbin Institute of Technology, Harbin, China, 150001}
	
	\subjclass[2020]{20E25, 43A07, 20E26, 37B05}
	\keywords{Topological full groups, Amenability, LEF groups}
	
	\date{\today}

	\begin{abstract}
	In this paper, we prove that for a topologically free distal action $\alpha$ of a countable discrete group $G$ on the Cantor set, the topological full group $[[\alpha]]$ is amenable if and only if $G$ is amenable. This generalizes a result by Cortez and Medynets. In addition, we obtain negative results on the amenability of topological full groups for actions with zero topological entropy by calculating the topological entropy of certain examples provided by Elek and Monod.
Furthermore, we demonstrate that the topological full group $[[\alpha]]$ of a minimal topologically free residually finite action $\alpha$ on the Cantor set is locally embeddable in the class of finite groups (LEF). This extends a result previously obtained by Grigorchuk and Medynets in the case of minimal $\Z$-actions. As an application, we show that topological full groups of certain Toeplitz subshifts on free groups are LEF and therefore sofic.
	\end{abstract}
	\maketitle

 \section{Introduction}
 Let $\alpha: G\curvearrowright X$ be a continuous action of a countable discrete group $G$ on the Cantor set $X$. The \textit{topological full group} of $\alpha$, denoted by $[[\alpha]]$, consists of all homeomorphisms $\varphi: X\to X$ for which there exists a continuous map $c(\varphi): X\to G$ such that $\varphi(x)=c(\varphi)(x)\cdot x$ for any $x\in X$. Here, $c(\varphi)$ is called a \textit{continuous orbit cocycle}. Equivalently, this means that there are a clopen partition of $X$ by $X=\bigsqcup_{i=1}^nX_i$ and group elements $\gamma_1,\dots, \gamma_n\in G$ such that $\varphi(x)=\gamma_i\cdot x$ for any $x\in X_i$. Topological full groups were originally introduced as an algebraic invariant to classify topological orbit equivalence relations (see, e.g., \cite{G-P-S}), and have since been generalized to the groupoid setting and also found many applications in geometric group theory. See more in, e.g., \cite{El}, \cite{G-M}, \cite{J-M}, \cite{Matui2}, \cite{Ne} and their introductions.    
  This class of groups has been verified as a new source for producing interesting groups with certain algebraic or geometric properties. For instance, in \cite{J-M}, for a minimal $\Z$-action $\alpha$ on the Cantor set,
Juschenko and Monod proved that the topological full group $[[\alpha]]$ is amenable. Combining with an earlier result obtained by Matui on minimal subshifts in \cite{Matui}, the derived subgroup $D[[\alpha]]$  actually provides the first example of a finitely generated infinite amenable simple group. Using topological full groups,  in \cite{Matt}, Matte Bon found the first examples of finitely generated simple groups with the Liouville property. Nekrashevych then produced in \cite{Ne1} the first examples of finitely generated simple groups with intermediate growth.

However, the problem of determining which dynamical systems possess amenable topological full groups remains a significant challenge. This question is also closely connected to various other open problems in dynamics and group theory, such as Katok's IET problem (refer to \cite{J-M-M-S} for further details).
It is worth noting that the examples presented in \cite{J-M} primarily focus on subshifts, which correspond to expansive actions. 
On the other hand, an example was constructed by Elek and Monod in \cite{E-M} presenting a minimal free $\Z^2$-subshift whose topological full group contains a free subgroup and, consequently, is non-amenable.

In contrast to subshifts, Cortez and Medynets demonstrated in \cite{C-M} that the topological full group of a minimal free equicontinuous action is amenable, provided that the acting group is amenable. They achieved this by reducing to free exact odometer actions (see \cite[Theorem 2.7]{C-M}). Motivated by this result and the fact that all equicontinuous actions are distal and all distal actions have zero topological entropy (see, e.g., \cite[Corollary 12.25]{K-L}), Yongle Jiang, Hanfeng Li, and Guohua Zhang posed the following question to the author.

\begin{ques}(Jiang, Li, Zhang)
Let $\alpha: G\curvearrowright X$ be a dynamical system of an amenable group $G$ on the Cantor set $X$. Suppose $\alpha$ is distal or has topological entropy $h_{\operatorname{top}}(X, G)=0$. Then, is the topological full group $[[\alpha]]$ still amenable? 
\end{ques}
 
We completely answer this question in this paper for topologically free actions. We require neither minimality nor full freeness. We remark that our assumption is quite mild because otherwise, the choice of continuous orbit cocycle for an element in the topological full group could be highly non-unique. 
Specifically, we prove a positive result for distal actions in Theorem \ref{thm: main 1} below.

\begin{customthm}{A}[Theorem \ref{thm: equi-amenable}]\label{thm: main 1}
		Let $\alpha: G\curvearrowright X$ be a distal topologically free action of an infinite countable discrete group $G$ on the Cantor set $X$. Then $[[\alpha]]$ is amenable if and only if $G$ is amenable.
	\end{customthm}

In addition, we establish negative results concerning the amenability of topological full groups for actions with zero topological entropy by calculating the topological entropy of certain examples in \cite{E-M}. This also implies that it is impossible to further generalize Theorem \ref{thm: main 1} to the entropy zero case.

 \begin{customthm}{B}[Theorem \ref{thm: non-amenable}]\label{thm: main 3}
There exists a minimal free action $\alpha: \Z^2\curvearrowright X$ on the Cantor set $X$ with $h_{\operatorname{top}}(X, \Z^2)=0$ (e.g., certain Elek-Monod's examples in \cite{E-M}) such that $[[\alpha]]$  contains a free group.
\end{customthm}

Another motivation of this paper is to investigate and identify more examples of LEF (locally embeddable in the class of finite groups) topological full groups. The concept of LEF groups was initially introduced by Gordon and Vershik in \cite[Definition 1]{G-V}. A group $G$ is classified as LEF if, for any finite set $H \subset G$, there exists a finite set $K$ with $H \subset K \subset G$, along with a binary operation $\odot: K^2 \to K$ that forms a group $(K, \odot)$, and such that for any $h_1, h_2 \in H$, the equality $h_1 \cdot h_2 = h_1 \odot h_2$ holds. All LEF groups are sofic.

 To the best knowledge of the author, the only known non-trivial examples of LEF topological full groups of Cantor dynamics are derived from minimal $\Z$-actions. This result was initially established by Grigorchuk and Medynets in \cite{G-M}, and alternate proofs can be found in \cite{Cor} and  \cite{El}.  To expand this list, we demonstrate that the topological full groups of minimal topologically free \textit{residually finite actions} (as defined in Definition \ref{defn: residually finite action}) on the Cantor set, are always LEF. Residually finite actions were introduced by Kerr and Nowak in \cite{K-N} and have many connections with the structure theory of $C^*$-algebras. 
	
	\begin{customthm}{C}[Theorem \ref{thm: topo group of residually finite action}]\label{thm: main 2}
		Let $\alpha: G\curvearrowright X$ be a minimal topologically free residually finite action of a countable discrete group on the Cantor set. Then $[[\alpha]]$ is LEF and thus sofic.
	\end{customthm}
	
	It is worth noting that Pimsner's result \cite[Lemma 2]{Pim} and \cite[Proposition 7.1]{K-N} imply that any minimal $\Z$-action on the Cantor set is residually finite. As a result, our Theorem \ref{thm: main 2} can be viewed as a generalization of the result obtained by Grigorchuk and Medynets in \cite{G-M}.
Furthermore, beyond $\Z$-actions, there are numerous other examples of residually finite actions provided in \cite{K-N}, which makes our Theorem C yield new examples in this direction. 

 \begin{cor}(Corollary \ref{cor: amenability and LEF})\label{cor: final cor}
 Let $\alpha: \F_r\curvearrowright X$ ($r\in \N\cup\{\infty\}$) be a minimal topologically free action such that there exists a $\F_r$-invariant Borel probability measure on $X$. Then $[[\alpha]]$ is LEF. In particular, there exists a minimal free $\F_r$-Toeplitz subshift $\beta$ whose topological full group $[[\beta]]$ is LEF and thus sofic.
\end{cor}

 It was shown in \cite[Corollary 7.6]{M} in the groupoid setting that the topological full groups of second countable minimal fiberwise amenable ample groupoids with a free unit (e.g., almost finite groupoids in \cite{Matui2}) are sofic. In particular, this implies that the group $[[\alpha]]$ is sofic whenever $\alpha$ is a minimal topologically free action by an amenable group $G$. On the other hand, to the best knowledge of the author, the alternating subgroup $\CA(X, \F_r)$ (see \cite{Ne} for the definition) of $[[\beta]]$ constructed from the $\F_r$-Toeplitz subshift $\beta$ in Corollary \ref{cor: final cor} is the first example of a finitely generated simple sofic topological full group arising from minimal free actions with non-amenable acting groups. 

In \cite[Theorem 6.7]{Matui2}, Matui demonstrated that the topological full group of a one-sided irreducible shift of finite type possesses the Haagerup property. We observe that, using the same argument by Cortez and Medynets in \cite{C-M}, if $\alpha$ is a minimal free equicontinuous action by a group with the Haagerup property, then the topological full group $[[\alpha]]$ also exhibits the Haagerup property.  We refer to Remark \ref{rmk: Haggerup property} for more details.

\subsection*{Outline of the paper:} 
The paper is organized as follows. In section 2, we mainly  establish Theorem \ref{thm: main 1}.  In section 3, we calculate the topological entropy of Elek-Monod's example in \cite{E-M} and prove Theorem \ref{thm: main 3}. In section 4, we investigate LEF topological full groups and prove Theorem \ref{thm: main 2}.

\section{Amenability of topological full groups for distal actions}
Throughout this section, let $G$ denote an infinite countable discrete group, $X$ the Cantor set and  $\alpha:G\curvearrowright X$ an action of $G$ on $X$. For $\gamma\in [[\alpha]]$, we denote by $\supp^o(\gamma)$ the \textit{open support} of $\gamma$, i.e., $\supp^o(\gamma)=\{x\in X: \gamma(x)\neq x\}$. We also define  the \textit{support} of $\gamma$ by $\supp(\gamma)=\overline{\supp^o(\gamma)}$.  We say a continuous action $\alpha: G\curvearrowright X$ is \textit{topologically free} if for any $g\neq e_G$, the set $\{x\in X: g\cdot x= x\}$ has empty interior. If $\alpha$ is topologically free, then  the set of free points for $\alpha$ is dense in $X$.

 \begin{rmk}\label{rem: unique cocycle}
Let $\alpha: G\curvearrowright X$ be a topologically free continuous action on the Cantor set and let $U$ be an open set in $X$. Then, if $\gamma\in [[\alpha]]$ sastifies $\gamma|_U=g$ and $\gamma|_U=h$ for some $g, h\in G$, then $g=h$ holds necessarily. Indeed, because the free points are dense, one has $\gamma(x)=g\cdot x$ and $\gamma(x)=h\cdot x$ for a free point $x\in U$. This implies that $g=h$ holds.  Furthermore, if $x$ is a free point, then for any $\gamma\in [[\alpha]]$, there is a unique $g\in G$ such that $\gamma(x)=g\cdot x$. In the case that $\alpha$ is minimal, the action $\alpha$ has a free point if and only if $\alpha$ is topologically free.
\end{rmk}

We now investigate topological full groups of distal actions. An action $\alpha: G\curvearrowright X$ is called
\textit{equicontinuous} if for any $\delta>0$ there exists an $\epsilon>0$ such that $d(x, y)\leq \epsilon$ implies $d(g\cdot x, g\cdot y)<\delta$ for any $g\in G$. The action $\alpha$ is  said to be \textit{distal} if $\inf_{g\in G}d(g\cdot x,g\cdot y)>0$ whenever $x\neq y$ in $X$. All equicontinuous actions are distal but the converse is not true in general. Nevertheless, a celebrated theorem due to Auslander, Glasner, and Weiss asserts that, whenever the acting group 
$G$ is finitely generated, equicontinuity is equivalent to distality for the corresponding actions. See \cite{A-G-W} and references therein. 

\begin{lem}\label{lem: distal}
	Let $\alpha: G\curvearrowright X$ be a distal action on the Cantor set $X$ and $T\subset [[\alpha]]$ a finite symmetric set. For each $\varphi\in T$, choose a continuous orbit cocycle $c(\varphi):X \to G$. Write $S=\bigcup_{\varphi\in T}\operatorname{ran}(c(\varphi))\subset G$, where $\operatorname{ran}(c(\varphi))$ is the range of $c(\varphi)$ in $G$, and denote by  $H=\langle S\rangle$.
 Then there is a finite clopen partition $\CU$ of $X$  satisfying the following.
\begin{enumerate}[label=(\roman*)]
    \item For any $\gamma\in \langle T\rangle$ and $U\in \CU$, there exists a $g\in H$ such that $\gamma|_U=g$.
    \item For any $\gamma\in \langle T\rangle$ with $\gamma\neq \id_X$, there is a $U\in \CU$ such that there exists $g\in H\setminus \{e_G\}$ with $\gamma|_U=g$.
    \end{enumerate}
\end{lem}
\begin{proof}
Note that $H=\langle S\rangle$ is a finitely generated subgroup of $G$. Now restrict the action $\alpha$ to the action $\alpha':H\curvearrowright X$, which is still distal and thus actually equicontinuous by \cite[Corollary 1.9]{A-G-W}.
	
With respect to $S$, there is a finite clopen partition  $\CP$ of $X$ such that for any $V\in \CP$ and $\varphi\in T$ there is a $g_\varphi\in S$ such that $\varphi=g_\varphi$ on $V$. Let $\delta>0$ be the Lebesgue number for the partition $\CP$. Since $\alpha'$ is equicontinuous, there is an $\epsilon>0$ such that  $d(x, y)\leq \epsilon$ implies $d(g\cdot x, g\cdot y)<\delta$ for any $g\in H$ and $x, y\in X$. Now choose another clopen partition $\CU$ of $X$ finer than $\CP$ such that each  $U\in \CU$ satisfies $\diam(U)<\epsilon$. Let $\gamma\in \Gamma= \langle T\rangle$. Write $\gamma=\psi_1\dots\psi_m$ in which each $\psi_j=\varphi_i$ for some $1\leq i\leq n$. 
	
For (i), we use induction on the length of $\gamma$. If $\gamma=\psi_1$, then because $\CU$ is finer than $\CP$ and $\psi_1\in T$, the definition of $\CP$ implies that for any $U\in \CU$ there is a $g_{\gamma}\in S\subset H$ such that $\gamma=g_{\gamma}$ on $U$. Now suppose this holds for all elements in $\Gamma$ with length $n$. Let $\gamma\in \Gamma$ be of length $n+1$, i.e., $\gamma=\psi\gamma'$ in which $\gamma'$ is of length $n$. Let $U\in \CU$. By induction, there is a $g_{\gamma'}\in H$ such that $\gamma'=g_{\gamma'}$ on $U$ and thus $\gamma'(U)$ is of diameter less than $\delta$ by our choice of $\epsilon$. This implies that $\gamma'(U)$ is contained in some member of $\CP$ and thus there is another $g_{\psi}\in S\subset H$ such that $\psi=g_{\psi}$ on $\gamma'(U)=g_{\gamma'}(U)$.  Now denote by $g_\gamma=g_{\psi}g_{\gamma'}\in H$ and one has $\gamma=g_\gamma$ on $U$. This establishes (i).

 For (ii), let $\gamma\in \langle T\rangle\setminus \{\id_X\}$. Then, there exists an $x\in X$ such that $\gamma(x)\neq x$. Take a finite clopen partition $\CU$ for $X$ satisfying (i). Then take $U\in \CU$ such that $x\in U$ and by (i), there exists a $g\in H$ such that $\gamma|_U=g$, which implies that $\gamma(x)=g\cdot x\neq x$. This implies that $g\neq e_G$. This establishes (ii).
\end{proof}

\begin{rmk}
    The author is grateful to a referee of a journal for pointing out the following more straightforward proof of Theorem \ref{thm: equi-amenable} rather than the author’s original approach based on establishing the hyperfiniteness of the sofic approximation graph sequences of the topological full groups.
\end{rmk}

Let $G$ be a countable discrete group. We refer to \cite[Definition 2.4]{C-M} for the definition of $G$-odometers and exact $G$-odometers. 

\begin{thm}\label{thm: equi-amenable}
		Let $\alpha: G\curvearrowright X$ be a topologically free distal action of an infinite countable discrete group $G$ on the Cantor set $X$. Then $[[\alpha]]$ is amenable if and only if $G$ is amenable.
	\end{thm}
	\begin{proof}
	The ``only if '' direction is trivial and thus it suffices to show the ``if'' part. Let $T$ be a finite symmetric set in $[[\alpha]]$ and denote by $\Gamma=\langle T\rangle$.  Write $S=\bigcup_{\varphi\in T}\operatorname{ran}(c(\varphi))\subset G$ and let $H=\langle S\rangle$. Then let $\CU=\{U_1, \dots, U_n\}$ be a finite clopen partition as in Lemma \ref{lem: distal} for $H$. Since $\alpha$ is topologically free, for each $1\leq i\leq n$, one may choose a free point $x_i\in U_i$ for the action of $G$ and thus a free point for $H$. Define $X_i=\overline{H\cdot x_i}$. Then because $H\curvearrowright X$ is equicontinuous, it follows from \cite[Proposition 7.27]{K-L} that the restricted action $\alpha_i: H\curvearrowright X_i$ is minimal. Moreover, the restriction of $H\curvearrowright X$ to $H\curvearrowright X_i$ yields a natural group homomorphism $\pi_i: \Gamma \to [[\alpha_i]]$ and thus one obtains a homomorphism:
    \[\pi: \Gamma\to \prod_{i=1}^n[[\alpha_i]]\]
    by $\pi(\gamma)=(\pi_1(\gamma),\dots, \pi_n(\gamma))$. Now, let $\gamma\neq \id_X$ in $\Gamma$. Then by the choice of $\CU$, Lemma \ref{lem: distal} implies that there has to be a $U_i\in \CU$ on which $\gamma|_{U_i}\equiv h$ for some $h\in H\setminus \{e\}$. This further entails that $\pi_i(\gamma)(x_i)=h\cdot x_i\neq x_i$ as $x_i$ is a free point for $H$. Thus,  $\pi(\gamma)$ is not the identity element as $\pi_i(\gamma)$ is not. Therefore, the map $\pi$ embeds $\Gamma$ into $\prod_{i=1}^n[[\alpha_i]]$.

    On the other hand, since each $\alpha_i: H\curvearrowright X_i$ is minimal and equicontinuous, one may assume $\alpha_i$ is an $H$-odometer by \cite[Corollary 2.8]{C-M} and thus has an exact $H$-odometer extension $\beta_i: H\curvearrowright Y_i$ by \cite[Proposition 1]{C-P}. Moreover, $\beta_i$ is topologically free because $\alpha_i$ is. But note that stabilizers of points for an exact odometer are all the same (see, e.g., \cite[Section 2.1]{C-P}). This shows that each $\beta_i$ is actually free and thus the topological full groups $[[\beta_i]]$ are amenable by \cite[Proposition 4.6]{C-M}. Finally, note that $[[\alpha_i]]$ can be embedded into $[[\beta_i]]$ as a subgroup (see, e.g., \cite[Lemma 4.8]{C-M}), from which one obtains that $\Gamma$ can be embedded into $\prod_{i=1}^n [[\beta_i]]$ and is thus amenable.  This further implies that $[[\alpha]]$ itself is amenable.
    \end{proof}

We end this section with the following elementary result.

\begin{prop}
	Let $\alpha: G\curvearrowright X$ be an action of a locally finite group $G$ on the Cantor set $X$. Then $[[\alpha]]$ is locally finite.
\end{prop}
\begin{proof}
	Let $T$ be a finite symmetric subset of $[[\alpha]]$. For each $\varphi\in T$, choose a continuous orbit cocycle $c(\varphi):X \to G$. Then define a finite set $S=\bigcup_{\varphi\in T}\operatorname{ran}(c(\varphi))\subset G$. 
	Let $H=\langle S\rangle$, which is a finitely generated subgroup of $G$ and thus a finite group. Now restrict the action $\alpha$ to the action $\alpha':H\curvearrowright X$, which is thus equicontinuous.
	
	Then the exact same argument as in Lemma \ref{lem: distal} shows that there is a clopen partition $\CO=\{O_1, \dots,  O_k\}$ of $X$ such that for any $\gamma\in \Gamma=\langle T\rangle$ and $O_i\in \CO$ there is a $g_{\gamma, i}\in H$ such that $\gamma|_{O_i}=g_{\gamma, i}$. This implies that $\Gamma$ is finite because $H$ is finite. 
\end{proof}

\section{Non-amenability of  topological full groups}
In this section, we establish Theorem \ref{thm: main 3}. The authors of \cite{E-M} constructed a minimal free $\Z^2$-subshift $\alpha$ whose topological full group $[[\alpha]]$ contains a copy of $\Z_2*\Z_2*\Z_2$ and is thus non-amenable. We will show that such a $\Z^2$-subshift may have zero topological entropy. For the reader's convenience, we recall necessary definitions and constructions in \cite{E-M}.

 Let $\CE$ be the set of all edges of the two-dimensional Euclidean lattice in $\R^2$, i.e., the Cayley graph of $\Z^2$.  There is a natural free translation action by $\Z^2$ on $\CE$, which yields a full shift action $\Z^2\curvearrowright Z=\{A, B, C, D, E, F\}^\CE$ defined by $(g\cdot \sigma)(e)=\sigma(g^{-1}e)$.

 \begin{defn}
   An edge-coloring $\sigma\in X$ is called \textit{proper} if the edges adjacent to any vertex $x\in \Z^2$ are colored differently.  
 \end{defn}

Denote by $\Sigma$ the subset of $Z$ consisting of all proper edge-colorings, which is a closed $\Z^2$-invariant subset of $Z$.
To each $x\in \{A, B, C, D, E, F\}$, there corresponds a continuous involution of $\Sigma$, which is denoted by the same letter and is defined in the following way on $\sigma\in \Sigma$:  if the vertex $(0,0)$ is connected to one of its four neighborhoods $v$ by an edge labeled by $x$, then $v$ is uniquely determined and $x(\sigma)$ will be the coloring obtained from translating  $\sigma$ towards $v$ (i.e., the origin is now where $v$ was). Otherwise, define $x(\sigma)=\sigma$. From the definition, it follows that $x\in [[\Z^2\curvearrowright \Sigma]]$. Therefore, there is a homomorphism $\varphi$ from $\langle A\rangle *\cdots * \langle F\rangle$ to $[[\Z^2\curvearrowright \Sigma]]$ and the free product preserves any $\Z^2$-invariant subset of $\Sigma$.  Define $\Delta=\langle A\rangle *\langle B\rangle*\langle C\rangle$.

A \textit{pattern} of a coloring $\sigma\in \Sigma$ is the isomorphism class of a finite labeled subgraph of $\sigma$. 
Now, one defines a particular proper coloring $\sigma$ as follows. One first enumerates non-trivial elements of $\Delta=\{w_i: i\in \N_+\}$ such that the length $|w_i|$ of $w_i$ satisfies $|w_i|\leq i$. This is possible as one may enumerate $\Delta$ by
\[\Delta=\{A, B, C, AB, AC, BC, BA, CA, CB, ABC, \dots\}\]
Then label integers $\Z$ with natural numbers in the following way: for each $n\in \N$, there is $g(n)$ such that any interval in $\Z$ of length $g(n)$ contains at least one element labeled by $n$. A concrete example is defined as follows. For any $m\in \Z\setminus\{0\}$, write $m=\pm 2^np$ for some odd number $p$. Then define the labeling function $L$ at $m$ by $L(m)=n$. Observe that $\{m\in \Z: L(m)=n\}=\{2^np: p\text{ is odd}\}$  and $2^n(p+2)-2^np=2^{n+1}$. This implies that $\{m\in \Z: L(m)=n\}$ is $2^{n+1}$-dense in $\Z$.

Using the labeling function $L$, we construct a specific proper edge-coloring $\sigma\in \Sigma$ as follows. Let $w=w_i$ be the $i$-th word in the enumeration of $\Delta$ above. Look at the vertical vertex lines $(v, \cdot)$ in the lattice $\Z^2$ such that $v$ is labeled by $i$, i.e., $L(v)=i$. Color those vertical lines in the following way. Starting at the point $(v, 0)$, copy the string $w$ onto the half-line above, beginning from the right end of $w$ (i.e., write $w^{-1}$ upwards). Color the next vertical upward edge by $D$ and then copy the string $w$ in the same way again. Now, repeat the process. Also continue the process below the vertex $(v, 0)$ to obtain a periodic coloring of the whole vertical line. Repeating the process for all non-trivial words in $\Delta$, we color all vertical lines $(v, \cdot)$ for all even $v\in \N_+$. Then color all lines $(v, \cdot)$ for either $v=0$ or $v$ odd in the same periodic way by $A$ and $D$. Finally, color all horizontal lines periodically with colors $E$ and $F$ uniformly by coloring edge $((m, n), (m+1, n))$ for $n\in \Z$ by $E$ if $m$ is even and by $F$ if $m$ is odd.

\begin{rmk}\label{rem: EM example}
    The example of Elek-Monod in \cite{E-M} is a topological full group $[[\alpha]]$ of a minimal free action $\alpha: \Z^2\curvearrowright M$, where $M$ is a minimal subsystem of the orbit closure, denoted by $\Omega(\sigma)$, of the proper coloring $\sigma\in \Sigma$ constructed above. The group $[[\alpha]]$ contains a free group because in this case the homomorphism $\varphi: \Delta \to [[\alpha]]$ above is verified to be injective.
\end{rmk}

We then calculate the topological entropy of $\Z^2\curvearrowright X:=\Omega(\sigma)$.
For $n\in \N_+$ write $F_n=[1, 2^n]\times [1, 2^n]$ for simplicity and note that $\{F_n: n\in \N_+\}$ forms a F{\o}lner sequence for $\Z^2$. Denote by $e_0, e_1$ the edges $((0, 0), (1, 0))$ and $((0, 0), (0, 1))$, respectively.
Define a pseudometric $\rho$ on $X$ by setting 
\[\rho(x, y)=
\begin{cases}
    0, \text{ if } x(e_0)=y(e_0) \text{ and } x(e_1)=y(e_1)\\
    1, \text{ otherwise.}
\end{cases}\]
First, $\rho$ is indeed a pseudometric. The only non-trivial part is the triangle inequality. Let $x, y, z\in X$. If $\rho(x, z)=0$ and $\rho(y, z)=0$, then $x(e_i)=z(e_i)=y(e_i)$ holds for any $i=0, 1$. Thus $\rho(x, y)=0$. This shows that $\rho(x, y)\leq \rho(x, z)+\rho(y, z)$.

In addition, it is obvious that $\rho$ is \textit{dynamically generating} in the sense that for any $x\neq y\in X$, there exists a $g\in \Z^2$ such that $\rho(gx, gy)>0$ (see \cite[Definition 9.35]{K-L}). Therefore, \cite[Theorem 9.38]{K-L} implies that the topological entropy \[h_{\operatorname{top}}(X, \Z^2)=h_{\operatorname{sep}}(\rho),\] where the definition of $h_{\operatorname{sep}}(\rho)$ is recalled below. Let $F$ be a finite subset of $\Z^2$. Define $\rho_F(x, y)=\max_{s\in F}\rho(sx, sy)$.  

\begin{defn}\cite[Definition 9.30, 9.31]{K-L}
    Let $F$ be a finite subset of $\Z^2$. A set $D\subset X$ is said to be $(\rho, F, \epsilon)$-\textit{separated} if $\rho_F(x, y)\geq \epsilon$ for all distinct $x, y\in D$. We write $\operatorname{sep}(\rho, F, \epsilon)$ for the maximum cardinality of a $(\rho, F, \epsilon)$-separated subset of $X$.
\end{defn}

\begin{defn}\cite[Definition 9.32]{K-L} Let $\{K_n: n\in \N\}$ be an arbitrary F{\o}lner sequence of $\Z^2$. Set 
    \[h_{\operatorname{sep}}(\rho, \epsilon)=\limsup_{n\to \infty}\frac{1}{|K_n|}\log \operatorname{sep}(\rho, K_n, \epsilon),\]
    and
    \[h_{\operatorname{sep}}(\rho)=\sup_{\epsilon>0}h_{\operatorname{sep}}(\rho, \epsilon).\]
\end{defn}

Now, we denote by
$\CD_n=\{e\in \CE: \text{end points of }e\in [-2^n, 0]\times [-2^n, 0]\}$ and write $\CD'_n=\CD_n\setminus\{e\in \CE: e\text{ is located on the x- or y-axis}\}$.
For our F{\o}lner sequence $\{F_n: n\in \N\}$ above, note that $\operatorname{sep}(\rho, F_n, \epsilon)=|\{x|_{\CD'_n}: x\in X\}|$ holds for any $1>\epsilon>0$ and thus
 \[\operatorname{sep}(\rho, F_n, \epsilon)\leq|\{x|_{\CD_n}: x\in X\}|=|\{\sigma|_{g\CD_n}: g\in \Z^2\}|\]
holds for any $1>\epsilon>0$ because $X=\Omega(\sigma)$. Therefore, by shifting $\CD_n$  to $\CE_n=\{e\in \CE: \text{end points of }e\in [0, 2^n]\times [0, 2^n]\}$, one has 
\[h_{\operatorname{top}}(X, \Z^2)\leq\limsup_{n\to \infty}\frac{1}{|F_n|}\log |\{\sigma|_{g\CE_n}: g\in \Z^2\}|.\]

It is thus left to evaluate $|\{\sigma|_{g\CE_n}: g\in \Z^2\}|$. Recall the labeling function $L$ above and we now work in the first quadrant. Note that $x=2^n$ is the smallest natural number that the word $w_n$ is pasted above the vertex $(x, 0)$. Thus, all the vertical words $w_i$ contained in the pattern $\sigma|_{\CE_n}$ have to satisfy $i\leq n$. The following is a key observation.

\begin{lem}\label{lem: key for entropy zero}
 Let $n\in \N_+$ and $m\in \N$. Write \[\CP_{n, m}=\{e\in \CE: \text{end points of }e\in [2^nm+1, 2^n(m+1)-1]\times [0, 2^n]\}.\] 
Then for any $n\in \N_+$, all  $\sigma|_{\CP_{n, m}}$ are of the same pattern for all $m\in \N$.
\end{lem}
\begin{proof}
    Note that any integer $k\in [2^nm+1, 2^n(m+1)-1]$ can be written as $2^nm+2^lp$ for some $l\in \N$ with $l<n$ and positive odd number $p$. Thus, $k=2^l\cdot(2^{n-l}m+p)$ in which $2^{n-l}m+p$ is an odd number. Therefore, $L(k)=l$ by our definition of the labeling function of $L$. This means the value of $L(k)$ does not depend on $m$. Instead, $L(k)$ is determined by the distance $|k-2^nm|$.
    Therefore, the definition of $\sigma$ shows that all $\sigma|_{\CP_{n, m}}$ are of the same pattern.
\end{proof}

\begin{prop}\label{prop: entropy zero}
   Let $\Z^2\curvearrowright X=\Omega(\sigma)$ be the subshift action above. Then one has $h_{\operatorname{top}}(X, \Z^2)=0$.
   \end{prop}
\begin{proof}
Let $n\in \N_+$ and $m\in \N$. Define \[\CE_{n, m}=\{e\in \CE: \text{end points of }e\in [2^nm, 2^n(m+1)]\times [0, 2^n]\},\] which adds two more vertical lines $(2^nm, \cdot)$ and $(2^n(m+1), \cdot)$ to $\CP_{n ,m}$ introduced in Lemma \ref{lem: key for entropy zero} for consideration.

Now we consider $\sigma|_{\CE_{n, m}}$ and the shifts of it upwards and downwards, i.e., all $\sigma|_{g\CE_{n, m}}$ for $g\in \{0\}\times \Z$. Note that  either $2^nm$ or $2^n(m+1)$ is labeled by $n$. For example, let's say $L(2^nm)=n$, i.e., $m$ is odd. Then there is a period of shifting the set
\[\CQ_{n, m}=\{e\in \CE: \text{end points of }e\in [2^nm, 2^n(m+1)-1]\times [0, 2^n]\}\]
upwards or downwards, which 
is at most $\prod_{1\leq i\leq n}(|w_i|+1)<\prod_{1\leq i\leq n}(i+1)=(n+1)!$. Now, because one uses four  colors to color vertical lines, i.e., $\{A, B, C, D\}$, one has the cardinality of the set of patterns 
\[|\{\sigma|_{g\CE_{n, m}}: n\in \N_+, m\text{ is odd, } g\in \{0\}\times \Z\}|\leq 4^{2^n}\cdot (n+1)!.\]
In the case $L(2^n(m+1))=n$, i.e., $m$ is even, the same argument shows that 
\[|\{\sigma|_{g\CE_{n, m}}: n\in \N_+, m\text{ is even, } g\in \{0\}\times \Z\}|\leq 4^{2^n}\cdot (n+1)!.\]
Combining these two cases, one has
\[|\{\sigma|_{g\CE_{n, m}}: n\in \N_+, m\in \N, g\in \{0\}\times \Z\}|\leq 2\cdot 4^{2^n}\cdot (n+1)!.\]

Then let $m\in \N_+$ and we look at the shifts $\CE_{n,m}$ to the right by $g=(k, 0)$, where $0<k<2^n$. This provides at most $2^n-1$ extra new patterns $\sigma|_{g\CE_{n, m}}$, depending on where the vertical line $(2^n(m+1),\cdot)$ is located in $g\CE_{n, m}$. On the other hand, except for the line  $(2^n(m+1),\cdot)$, other vertical lines in $g\CE_{n, m}$ have been labeled by a number strictly less than $n$ by the proof of Lemma \ref{lem: key for entropy zero}, which is independent of the value of $m$. 

Therefore, using the same argument as above, one counts all possible different vertical shifts of these $\sigma|_{g\CE_{n, m}}$ for $g=(k,0)$ with $0<k<2^n$ to obtain that
\[|\{\sigma|_{g\CE_{n,m}}: g\in \{k\}\times \Z\}|\leq 4^{2^n}\cdot n!\]
holds for any $m\in \N$. Thus in total one has
\[|\{\sigma|_{g\CE_n}: g\in \N\times \Z\}|\leq (2^n-1)\cdot 4^{2^n}\cdot n!+ 2\cdot 4^{2^n}\cdot (n+1)!< 2\cdot 2^n\cdot 4^{2^n}\cdot (n+1)!.\]
Recall that the distribution of the labeling on $\Z$ is mirror-symmetric to the $y$-axis. One thus has 
\[|\{\sigma|_{g\CE_n}: g\in (\N\sqcup \{k\in \Z: k\leq -2^n\})\times \Z\}|\leq 4\cdot 2^n\cdot 4^{2^n}\cdot (n+1)!.\]

Finally, we estimate the size of $I=\{\sigma|_{g\CE_n}: g\in (-2^n, -1]\times \Z\}$. Note that the square $\sigma|_{g\CE_n}\in I$ is located between two lines $(-2^n, \cdot)$ and $(2^n, \cdot)$, in which any even non-zero integer $k\in (-2^n, 2^n)$ has the label $L(k)<n$. Then similar to the argument above, the period for shifting such $\sigma|_{g\CE_n}$ in the vertical direction is at most \[\prod_{1\leq i<n}(|w_i|+1)<\prod_{1\leq i<n}(i+1)=n!.\] Therefore, taking all horizontal shifts $\sigma|_{g\CE_n}$ for $g\in (-2^n, -1]\times \{0\}$ into consideration, in total one has 
\[|\{\sigma|_{g\CE_n}: g\in (-2^n, -1]\times \Z\}|\leq 2^n\cdot n!.\]
Then, using Stirling's approximation for $n!$ and $(n+1)!$ (see, e.g., \cite[Lemma 10.1]{K-L}), one has
\[|\{\sigma|_{g\CE_n}: g\in \Z^2\}|\leq 2^n\cdot n!+ 4\cdot 2^n\cdot 4^{2^n}\cdot (n+1)!\leq 2^n\cdot n\cdot e\cdot (n/e)^n\cdot 4^{2^n}\cdot (4n+5),\]
which implies
\begin{equation*}
    \begin{split}
        \frac{\log|\{\sigma|_{g\CE_n}: g\in \Z^2\}|}{|F_n|}\leq \frac{n\log2+\log n+1+n(\log n-1)+\log (4n+5)+2^n \log 4}{2^n\cdot 2^n}\to 0
    \end{split}
\end{equation*}
as $n\to \infty$.
Therefore, one has $h_{\operatorname{top}}(X, \Z^2)=0$. 
\end{proof}
We thus have the following theorem.

\begin{thm}\label{thm: non-amenable}
    There exists a minimal free action $\alpha: \Z^2\curvearrowright M$ on the Cantor set $M$ with topological entropy zero such that $[[\alpha]]$  contains a free group.
\end{thm}
\begin{proof}
Let $X=\Omega(\sigma)$ be the orbit closure of the proper coloring $\sigma$.  Let $M$ be a $\Z^2$-minimal subsystem of $X$, whose topological full group $[[\alpha]]$ contains a free group by Remark \ref{rem: EM example}. One also has $h_{\operatorname{top}}(M, \Z^2)=0$ because
Proposition \ref{prop: entropy zero} has established $h_{\operatorname{top}}(X, \Z^2)=0$ and $M\subset X$.
\end{proof}
	
	\section{Residually finite actions and LEF topological full groups}
In this section, we establish Theorem \ref{thm: main 2}. Residually finite actions were introduced in \cite{K-N}. We recall below their definition in the case that the underlying space is perfect compact metrizable space, e.g.,the Cantor set. Let $\alpha: G\curvearrowright X$ be an action. Since there will be more than one action involved in this section, we write $\alpha(s)(x)$, instead of ``$s\cdot x$'', for the point  translated by $s$ from $x$ under the action $\alpha$.
	
	\begin{defn}\cite[section 2]{K-N}\label{defn: residually finite action}
		A continuous action $\alpha$ of $G$ on a perfect compact metrizable space $X$, equipped with a compatible metric $d$, is said to be \textit{residually finite} if for any finite $F\subset G$ and $\epsilon>0$, there is a finite set $E\subset X$, equipped with a $G$-action  $\beta$ such that $E$ is $\epsilon$-dense in $X$ and $d(\alpha(s)(z), \beta(s)(z))<\epsilon$ for all $z\in E$ and $s\in F$.
	\end{defn}

If a residually finite action $\alpha$ on the Cantor set is topologically free, we may arrange points in $E$ to be free points for $\alpha$ by a standard approximation argument.

\begin{prop}\label{prop: free point}
    Let $\alpha: G\curvearrowright X$ be a topologically free residually finite action on the Cantor set $X$. Then for any finite $F\subset G$ and $\epsilon>0$, there is a finite set $E\subset X$ consisting of free points for $\alpha$, equipped with a $G$-action  $\beta$ such that $E$ is $\epsilon$-dense in $X$ and $d(\alpha(s)(z), \beta(s)(z))<\epsilon$ for all $z\in E$ and $s\in F$.
    \end{prop}
    \begin{proof}
    Let $F$ be a finite set in $G$ and $\epsilon>0$.   Since $\alpha$ is residually finite, there exists a finite $E\subset X$ equipped with an action $\beta$, which is $\epsilon/3$-dense in $X$ and $d(\alpha(s)(z), \beta(s)(z))<\epsilon/3$ for all $z\in E$ and $s\in F$.  Now, we denote by $\epsilon_0=\min\{d(x, y): x, y\in E\}$. Moreover, because $F$ is finite and $X$ is compact, there exists a $\delta>0$ such that if $d(x, y)<\delta$, then $d(\alpha(s)(x), \alpha(s)(y))<\epsilon/3$ for any $x, y\in X$ and $s\in F$. Then because $\alpha$ is topologically free, for each $x\in E$, choose a free point $x'\in X$ for $\alpha$ such that $d(x, x')<\min\{\epsilon_0, \epsilon/3, \delta\}$. Note that the map $x\mapsto x'$ is injective. Define $E'=\{x': x\in E\}$ and an action $\beta'$ of $G$ on $E'$ by $\beta(s)(x')=y'$ whenever $\beta(s)(x)=y$ for $s\in G$. Then, by our construction, the set $E'$ is $2\epsilon/3$-dense in $X$. In addition, observe that $d(\beta'(s)(x'), \beta(s)(x))<\epsilon/3$ and
    $d(\alpha(s)(x), \alpha(s)(x'))<\epsilon/3$
    for any $s\in F$ by the definition of $\beta'$ and our choice of $\delta$. This implies the distance $d(\alpha(s)(x'),\beta'(s)(x'))$ is bounded by
    \begin{align*}
       d(\alpha(s)(x'), \alpha(s)(x))+d(\alpha(s)(x), \beta(s)(x))+d(\beta(s)(x), \beta'(s)(x))<\epsilon 
    \end{align*}
 for any $x'\in E'$ and $s\in F$. Thus, the set $E'$ and action $\beta'$ satisfy the requirement.
    
    \end{proof}
	
	The following is a characterization of LEF-groups in \cite{G-V}.
	
	\begin{lem}\cite[Theorem-definition]{G-V}\label{lem: charac of LEF}
		For countable groups $G$, the following are equivalent.
		\begin{enumerate}
			\item $G$ is an LEF-group.
			\item There exists a countable sequence of finite groups $H_n$ and a system of maps $\pi_n: G\to H_n$ such that
			\begin{enumerate}
				\item[a.] for any $x, y\in G$ with $x\neq y$ there exists an $N>0$ such that $\pi_n(x)\neq \pi_n(y)$ whenever $n>N$, and
				\item[b.] for any $x, y\in G$ there is an $N$ such that $\pi_n(x\cdot y)=\pi_n(x)\cdot \pi_n(y)$ whenever $n>N$. 
			\end{enumerate}
		\end{enumerate} 
	\end{lem}
	
	In our setting of topological full groups, we have the following criteria for LEF-ness by Lemma \ref{lem: charac of LEF}. Denote by  $\map(A)$ the set of all self-maps on $A$ and define the \emph{Hamming distance} on $\map(A)$ by $d_H(f, g)=|\{x\in A: f(x)\neq g(x)\}|/|A|$.
     Note that $\map(A)$ is not a group in general.
	
	\begin{lemma}\label{lemma: compressed LEF representation}
		Let $\{e_\Gamma=\gamma_0, \gamma_1,\dots, \}$ be an enumeration of a countable discrete group $\Gamma$. Suppose for any $i\geq 1$ there is a constant $\epsilon_i>0$ and for any $n\geq 1$ there is a map $\Theta_n: \Gamma\to \map(A_n)$ for some finite set $A_n$ with $\Theta_n(e_\Gamma)=\id_{A_n}$ and satisfying the condition that for all $r>0$  there exists $K_{r}>0$ such that  if $n>K_{r}$ one always has
		\begin{enumerate}
			\item $d_H(\Theta_n(\gamma_i\gamma_j), \Theta_n(\gamma_i)\Theta_n(\gamma_j))=0$ if $1\leq i, j\leq r$.
			\item $d_H(\Theta_n(\gamma_i), \id_{A_n})>\epsilon_i$ if $1\leq i\leq r$.
		\end{enumerate}
		Then $\Gamma$ is LEF.
	\end{lemma}
	\begin{proof}
	Let $x=\gamma_i$ and $x^{-1}=\gamma_j$ in the enumeration of $\Gamma$ with $i, j>0$. Then by assumption, for $i, j$, there exists a $K_x=\max\{K_i, K_j\}$ such that $d_H(\Theta_n(x), \id_{A_n})>0$ and $d_H(\Theta_n(e_\Gamma), \Theta_n(x)\Theta_n(x^{-1}))=0$ whenever $n>K_x$. These imply that if $n>K_x$, then $\Theta_n(x)\neq \id_{A_n}$  and
	$\Theta_n(x)\in \sym(A_n)$ satisfying $\Theta_n(x)^{-1}=\Theta_n(x^{-1})$ because $\id_{A_n}=\Theta_n(e_\Gamma)$.
 
 Define a finite group $H_n=\sym(A_n)$ and a map $\pi_n: \Gamma\to H_n$ in the following way. For any $x\in \Gamma$, define $\pi_n(x)=\Theta_n(x)$  if $n>K_x$. Otherwise, define $\pi_n(x)$ to be any element in $H_n$. Now, let $x\neq y\in \Gamma$ and for any $n>\max\{K_{x^{-1}y}, K_x, K_y\}$, observe
		\[\pi_n(x^{-1}y)=\Theta_n(x^{-1}y)=\Theta_n(x)^{-1}\Theta_n(y)=\pi_n(x)^{-1}\pi_n(y)\]
		and $\Theta_n(x^{-1}y)\neq \id_{A_n}$. This implies $\pi_n(x)\neq \pi_n(y)$.  Moreover,  one has 
		\[\pi_n(xy)=\Theta_n(xy)=\Theta_n(x)\Theta_n(y)=\pi_n(x)\pi_n(y)\]
	whenever $n>\max\{K_{xy}, K_x, K_y\}$. Thus, $\Gamma$ is LEF.
	\end{proof}
	
	Let $\alpha: G\curvearrowright X$ be a minimal topologically free residually finite action of a countable discrete group $G$ on the Cantor set $X$.  Fix an enumeration $\{\gamma_0, \gamma_1,\dots \}$ of $[[\alpha]]$ with $\gamma_0=\id_X$. Let $M_n=\{\gamma_0,\dots, \gamma_n\}$ and $\{\eta_n>0: n\in \N\}$ be a decreasing sequence converging to zero. For each $\gamma_i$, we choose a continuous orbit cocycle $c(\gamma_i): X\to G$, which yields a clopen partition $\CU_i=\{U_1, \dots, U_{l_i}\}$ of $X$ and a collection of group elements $\CC_i=\{g_{i, 1},\dots, g_{i, l_i}\}\subset G$ such that $\gamma_i|_{U_j}=g_{i, j}$ for any $j\leq l_i$. Define $F_n=\bigcup_{i\leq n}\CC_i$. 
 
 
 Since $\alpha$ is residually finite, for each $F_n$ and $\eta_n$, Proposition \ref{prop: free point} implies that there exists a finite set $E_n\subset X$ consisting of free points  for $\alpha$ such that $E_n$ is $\eta_n$-dense in $X$ and is equipped with an action $\beta_n: G\curvearrowright E_n$ satisfying $d(\beta_n(g)(y), \alpha(g)(y))\leq \eta_n$ for any $g\in F_n$ and $y\in E_n$. By definition, $\{F_n: n\in \N\}$ is an increasing sequence in $G$. Moreover, because $G$ can be viewed as a subgroup of $[[\alpha]]$ and the action $\alpha$ is topologically free, Remark \ref{rem: unique cocycle} implies that $G=\bigcup_{n=1}^\infty F_n$. 
 
 Define a probability measure $\mu_n=(1/|E_n|)\sum_{z\in E_n}\delta_z$. Then, \cite[Proposition 2.3]{K-N} implies that a weak*-limit point $\mu$ of $\{\mu_n: n\in\N\}$ is a $G$-invariant Borel probability measure. By passing to a subsequence $\{\mu_{k_n}\}$, we may assume $\mu_n\to \mu$ under weak*-topology. To simplify the notation in the proof of the following theorem,  we denote by $D_n=M_{k_n}$ and write $E_n$ for $E_{k_n}$, and $F_n$ for $F_{k_n}$, as well as $\{\eta_n\}$ for $\{\eta_{k_n}\}$.

	\begin{thm}\label{thm: topo group of residually finite action}
		Let $\alpha: G\curvearrowright X$ be a minimal topologically free residually finite action of a countable discrete group $G$ on the Cantor set $X$.  Then its topological full group $[[\alpha]]$ is LEF.
	\end{thm}
	\begin{proof}
		Fix an enumeration of $[[\alpha]]$ by $\{\id_X=\gamma_0, \gamma_1,\dots\}$. We recall the notation $D_n=\{\gamma_0, \gamma_1, \dots, \gamma_{k_n}\}$, the  set $E_n\subset X$ consisting of free points for $\alpha$,  the set $F_n=\bigcup_{i\leq k_n}\CC_i$, and the decreasing sequence $\{\eta_n: n\in \N\}$ converging to $0$ above, such that the following hold.
\begin{enumerate}
    \item The  measure
  $\mu_n=(1/|E_n|)\sum_{z\in E_n}\delta_z\to \mu$
for a $G$-invariant probability measure $\mu$ under weak*-topology. 
\item The set $E_n$ is $\eta_n$-dense in $X$ and is equipped with an action $\beta_n: G\curvearrowright E_n$ satisfying $d(\beta_n(g)(x), \alpha(g)(x))\leq \eta_n$ for any $g\in F_n$ and $x\in E_n$.
\end{enumerate}
  
     For any $\gamma_i\in [[\alpha]]$ for $i\geq 1$, i.e.,  $\gamma_i\neq \id_X$, choose a non-empty clopen set $A_i\subset \supp^o(\gamma_i)$ such that $\gamma_i(A_i)\cap A_i=\emptyset$. Then since $\alpha$ is minimal, one has $\mu(A_i)=2\epsilon_i$ for some $\epsilon_i>0$, which implies that there is an $N_i>0$ such that $\mu_n(A_i)>\epsilon_i$ whenever $n>N_i$.
		Now, define $\Theta_n: [[\alpha]]\to \map(E_n)$ by 
 \[\Theta_n(\gamma_i)(x)=\beta_n(g)(x),\]
 where $g$ is the unique element in $\CC_i$ such that $\gamma_i(x)=\alpha(g)(x)$. Note that for each $n\in \N$, since $E_n$ consists of free points, Remark \ref{rem: unique cocycle} implies that for any $\gamma\in [[\alpha]]$ and $x\in E_n$, there is a unique $g\in G$ independent of the choice of orbit cocycles such that $\gamma(x)=\alpha(g)(x)$. Thus, for any $x\in E_n$, if $\id_X(x)=\alpha(g)(x)$, then $g=e_G$ holds necessarily. This implies that $\Theta_n(\id_X)(x)=\beta_n(e_G)(x)=x$ for any $x\in E_n$ and thus  $\Theta_n(\id_X)=\id_{E_n}$. 
 
 For each $\gamma_i$, we recall the clopen partition $\CU_i= \{U_1, \dots, U_{l_i}\}$ and set of group elements $\CC_i=\{g_{i,1},\dots, g_{i, l_i}\}$ introduced above such that $\gamma_i(x)=\alpha(g_{i, j})(x)$ for any $x\in U_j$.  Write $d_i$ for the Lebesgue number of the cover $\CU_i$. We now verify that all such $\Theta_n$ satisfy  Lemma \ref{lemma: compressed LEF representation} above. Let $T_r=\{\gamma_0, \dots, \gamma_r\}$, and $\epsilon>0$. Choose $N>0$ such that  whenever $n>N$, one has \[\eta_n\leq \min\{d(\gamma_i(A_i), \gamma_i(A_i)^c), d_i: 1\leq i\leq r\}.\] Then, choose $K_{r}\geq \max\{N, r, N_i: 1\leq i\leq r\}$ and let $n>K_{r}$. First, $n>r$ implies $T_r\subset D_n$. Then, $n>\max\{N_i: 1\leq i\leq r\}$ implies
		\[\mu_n(A_i)=\frac{|E_n\cap A_i|}{|E_n|}\geq \epsilon_i\]
for any $1\leq i\leq r$.
		In addition, for each $1\leq i\leq r$, observe $\gamma_i(x)\neq x$ for any $x\in A_i\cap E_n$ by the choice of $A_i$. Let $x\in A_i\cap E_n$. There exists a $g\in \CC_i\subset F_n$ such that $\gamma_i(x)=\alpha(g)(x)$. Then, by our choice of $E_n$, one has $d(\alpha(g)(x), \beta_n(g)(x))<\eta_n$, which implies that $\beta_n(g)(x)\in \gamma_i(A_i)$ as well because the choice of $n>N$. This further implies that $\Theta_n(\gamma_i)(x)\neq x$ because $\Theta_n(\gamma_i)(x)=\beta_n(g)(x)$ and $\gamma_i(A_i)\cap A_i=\emptyset$ by definition of $A_i$. Therefore, one has 
		\[|\{x\in E_n: \Theta_n(\gamma_i)(x)\neq x\}|\geq |E_n\cap A_i|\geq \epsilon_i|E_n|\]
		and thus $d_H(\Theta_n(\gamma_i), \id_{E_n})\geq \epsilon_i$ for any $1\leq i\leq r$.
		
		Let $1\leq i, j\leq r$ and $z\in E_n$, which is a free point for $\alpha$. Then 
there are a $g\in \CC_j$ and a $U\in \CU_j$ such that $z\in U$ and $\gamma_j(z)=\alpha(g)(z)$. Moreover, for $\gamma_j(z)$, there are an $h\in \CC_i$ and a $U'\in \CU_i$ such that $\gamma_j(z)=\alpha(g)(z)\in U'$ and $\gamma_i\gamma_j(z)=\alpha(h)\alpha(g)(z)$. This actually implies that $\Theta_n(\gamma_i\gamma_j)(z)=\beta_n(hg)(z)$ because $z$ is a free point for $\alpha$ and thus there is a unique $s\in G$ such that $\gamma_i\gamma_j(z)=\alpha(s)(z)$ by Remark \ref{rem: unique cocycle}. 

On the other hand, by definition, one has $\Theta_n(\gamma_j)(z)=\beta_n(g)(z)$. Then by our choice of $n>K_{r}\geq N$, the fact $d(\beta_n(g)(z), \alpha(g)(z))<\eta_n\leq d_i$ implies that $\beta_n(g)(z)\in U'$ as well, and thus $\gamma_i(\beta_n(g)(z))=\alpha(h)(\beta_n(g)(z))$. By the fact that $\beta_n(g)(z)\in E_n$ is still a free point for $\alpha$ and $\beta_n$ is an action, one has 
		\[\Theta_n(\gamma_i)\Theta_n(\gamma_j)(z)=\Theta_n(\gamma_i)(\beta_n(g)(z))=\beta_n(h)(\beta_n(g)(z))=\beta_n(hg)(z)=\Theta_n(\gamma_i\gamma_j)(z).\]
	This means that $d_H(\Theta_n(\gamma_i\gamma_j), \Theta_n(\gamma_i)\Theta_n(\gamma_j))=0$ for any $1\leq i, j\leq r$. Therefore, Lemma \ref{lemma: compressed LEF representation} shows that $[[\alpha]]$ is LEF.
	\end{proof}

	\begin{rmk}
		It was shown in \cite[Proposition 7.1]{K-N} that a $\Z$-system generated by a single homeomorphism $\varphi$, is residually finite if and only if it is \textit{chain recurrent} (see, e.g., \cite[Section 7]{K-N} for the definition). On the other hand, Pimsner showed in \cite[Lemma 2]{Pim} that a $\Z$-system is chain recurrent if and only if there is no open set $U\subset X$ such that   $\varphi(\overline{U})$ is a proper subset of $U$. Therefore, any minimal $\Z$-action $\alpha$ on the Cantor set is residually finite, which implies that $[[\alpha]]$ is LEF by Theorem \ref{thm: topo group of residually finite action}. This is exactly the result obtained in \cite[Theorem 5.1]{G-M}. See also \cite[Theorem 2]{El}. Therefore, our Theorem \ref{thm: topo group of residually finite action} has generalized the result.
	\end{rmk}
	
	Besides minimal $\Z$-actions, there are many other examples of minimal residually finite actions on the Cantor set. For example, it was shown in \cite[Theorem 5.2]{K-N} that any minimal action of the free group $\F_r$ ($r\in \N\cup\{\infty\}$) on the Cantor set is residually finite if it admits an $\F_r$-invariant Borel probability measure (see \cite[Theorem 5.2]{K-N}.

 Then we have the following corollary based on the construction of certain Toeplitz subshifts $\beta$ of free groups in \cite[Theorem 1.2, 1.3]{C-C-G}. Note that Toeplitz subshifts are well-known to be minimal and can be chosen to be free by \cite[Remark 2.5]{C-C-G}. This provides more LEF examples of topological full groups. In particular, combining results in \cite{Ne}, the alternating subgroup $\CA(X, \F_r)$ (see the definition in \cite{Ne}) of $[[\beta]]$ constructed from this $\F_r$-Toeplitz subshift is the first example of a finitely generated simple sofic topological full group arising from minimal free actions with non-amenable acting groups.  
	
	\begin{cor}\label{cor: amenability and LEF}
   Let $\alpha: \F_r\curvearrowright X$ ($r\in \N\cup\{\infty\}$) be a minimal topologically free action such that there exists a $\F_r$-invariant Borel probability measure on $X$. Then $[[\alpha]]$ is LEF. In particular, there exists a minimal free $\F_r$-Toeplitz subshift $\beta$ whose topological full group $[[\beta]]$ is LEF and thus sofic.
    \end{cor}

We finally remark the following. In \cite[Theorem 6.7]{Matui2}, Matui showed that the topological full group of a one-sided irreducible shift of finite type has the Haagerup property. One actually can construct more topological full groups with the Haagerup property in the following elementary way.

 \begin{rmk}\label{rmk: Haggerup property}
 If $\alpha: G\curvearrowright X$ is a minimal free equicountinuous action of a countable discrete group $G$ with the Haagerup property, then its topological full group $[[\alpha]]$ has the Haagerup property as well. This is established by the same methods in \cite[Section 4]{C-M} by Cortez and Medynets because the Haagerup property is also preserved by taking (closed) subgroups, finite directed products, increasing unions, and group extensions by amenable groups.
 \end{rmk}

	\section{Acknowledgement}
	The author is indebted to Maksym Chaudkhari, Yongle Jiang, Kate Juschenko, David Kerr, Hanfeng Li, Ning Ma, Guohua Zhang, and Tianyi Zheng for their encouragement,  discussion and comments, careful reading of the early drafts, corrections, and informing the author of necessary references. The author would like to thank the anonymous referee for the very helpful comments and suggestions which improve the paper significantly.

\end{document}